\documentclass[11pt]{amsart}

\usepackage[utf8x]{inputenc}
\usepackage[english]{babel}
\usepackage[all,cmtip]{xy}

\usepackage{graphicx}
\usepackage[margin=2.8cm]{geometry}
\usepackage{amsfonts}
\usepackage{amsmath}
\usepackage{amssymb}
\usepackage{amsthm}
\usepackage{dsfont}
\usepackage{tikz}
\usepackage{pgfplots}
\usepackage{MnSymbol}

\usepackage[T1]{fontenc}

\usepackage{paralist}

\usepackage{hyperref}

\renewcommand\phi{\varphi}
\newcommand\N{\mathbb{N}}

\newcommand\R{\mathbb{R}}

\newtheorem{thm}{Theorem}[section]
\newtheorem*{thm*}{Theorem}

\newtheorem{prop}[thm]{Proposition}

\newtheorem{quest}[thm]{Question}

\theoremstyle{definition}

\newtheorem{example}[thm]{Example}
\newtheorem{rem}[thm]{Remark}

\title[On embeddings of finite subsets of $\ell_2$]{On embeddings of finite subsets of $\ell_2$}

\author{James Kilbane}
\address{Department of Pure Maths and Mathematical Statistics, University of Cambridge}
\email{jk511@cam.ac.uk}

\date{\today}

\begin{document}

\begin{abstract}
We study finite subsets of $\ell_2$, and more generally any metric space, and consider whether these isometrically embed into a Banach space. Our results partially answer a question of Ostrovskii \cite{ostrov}, on whether every infinite-dimensional Banach space contains every finite subset of $\ell_2$ isometrically.
\end{abstract}

\maketitle


\section{Introduction} \label{sec:intro}
The area of embedding metric spaces into Banach spaces has been an active area of research. Questions in this area tend to form the following type: Suppose I have a Banach space $X$ and a metric space lying in some class $\mathcal{Y}$. Is there an embedding with some 'good' property into $X$? A variety of formats of this question have been asked, eg, infinite discrete metric spaces with some expander property embedding into $L_1$, locally finite metric spaces whose finite subsets admit bilipschitz embeddings, etc.

This paper focuses on the following variant of this question: 
\begin{quest}\label{quest:u}
Suppose that $X$ is a Banach space and let $\mathcal{Y}$ be the class of finite metric spaces that admit $1+\epsilon$ distortion embeddings into $X$ for each $\epsilon > 0$. Then, for any $Y \in \mathcal{Y}$, does there exist an isometric embedding of $Y$ into $X$? Can we come up with criteria on the spaces $X$ such that this is true?
\end{quest}

This question attempts to capture a compactness principle for Banach spaces, which could broadly be phrased as "is there a way of passing to a limit?" 

The following related question was raised by Ostrovskii \cite{ostrov} on MathOverflow:
\begin{quest} \label{quest:v}
Suppose that $X$ is an infinite-dimensional Banach space, and let $\mathcal{Y}$ be the class of finite subsets of $\ell_2$. Then does any $Y \in \mathcal{Y}$ admit an isometric embedding into $X$?
\end{quest}

The purpose of this paper is to give some insights into both of these questions. For Question \ref{quest:v} we will establish the following result:
\begin{thm}
Suppose that $X$ is an infinite-dimensional Banach space and $Y$ is an affinely independent finite subset of $\ell_2$. Then $Y$ admits an isometric embedding into $X$.
\end{thm}

This will be proved in Section 3. Theorem 1.3 is, in some sense, a local result. The construction proceeds inductively: we embed the first $k$ points of $Y$ and then search for an embedding of the $k+1$'st point (see the proof of Theorem 3.1). We then show that such an inductive procedure (where the points become fixed each time) can not work in the case of affine dependence (see Theorem \ref{thm:counter}).

After the first edition of this preprint, it was pointed out to the author (by Lionel Nguyen Van Th\'{e}) that Theorem 3.1 was proven in a paper of Shkarin \cite{shkarin}. We thus make no claims of originality of the contents of Theorem 3.1, however the method of proof is different here.

In Section 4, we will discuss Question \ref{quest:u}. First of all we will show that known results imply that many spaces satisfy the conclusions of Question 1.1. After this we will show that a method similar to that of the proof of Theorem 3.1 allows us to show that any space with no non-trivial cotype contains every concave metric space (for the definition of concave see Section 2.1). 

We then discuss some counter-examples, and show that not all spaces satisfy the conclusion of Question 1.1. We can, in fact, generate an entire class of such spaces.

The results of Section 4 allow us to extend our results about Question 1.2, and show that several spaces contain all finite subsets of $\ell_2$ isometrically.

The paper concludes with an application of Theorem 3.1, some general remarks, and some open problems.
\section{Definitions and Notation}
\subsection{Metric Space Definitions} In this paper, we will mainly use the letters $X,Y,\dots$ for a metric space with corresponding metrics $d_X, d_Y, \dots$.

Suppose that $f:X \rightarrow Y$ is a map on metric spaces. We say that $f$ \emph{has distortion} $K$ if there is some $r > 0$ such that $r d_X(x,y) \leq d_Y(f(x),f(y)) \leq r K d_X(x,y)$.

If $(X,d)$ is a metric space, and $\alpha \in (0,1)$ then the \emph{$\alpha$-snowflake} of $X$ is defined to be the metric space $(X,d^\alpha)$.

The following definition is less standard: We say that a metric space is \emph{concave}\cite{weaver} if the triangle inequality is always strict, ie that for all distinct triples $x,y,z$, $d_Y(x,y) + d_Y(y,z) > d_Y(x,z)$. We remark that for finite metrics the concave metrics are 'dense' amongst them (in the sense of Gromov-Hausdorff distance). This is because if we have $d_Y$ a metric, the $\alpha$-snowflake $d_Y^\alpha$ for $\alpha \in (0,1)$ is a metric, and $d_Y^{1-1/n} \rightarrow d_Y$.

We say that a metric space is \emph{equilateral} if there exists some constant $K$ such that whenever $x \neq y$, $d(x,y) = K$.

If $x,y$ are distinct points in a metric space, we say that $z$ is \emph{a metric midpoint of $x$ and $y$} if $d(x,z) = d(z,y) = \frac{1}{2} d(x,y)$.

\subsection{Banach Space Definitions and Classical results}\label{prereq} Suppose that $X,Y$ are Banach spaces. We let $d(X,Y)$ be the \emph{Banach-Mazur distance} defined by $d(X,Y) = \inf \{\|T\| \|T^{-1}\| : T$ is an isomorphism from $X$ to $Y\}$.

The main classical result we will use is the Dvoretzky theorem which says (heuristically) that every infinite-dimensional Banach space contains $n$-dimensional slices that arbitrarily well look like $\ell_2^n$. More precisely,
\begin{thm}
Let $n \in \N$ and $\epsilon > 0$. Then there is some $N$ (depending on both $n$ and $\epsilon$) such that if $X$ is an $N$-dimensional Banach space, then there is a subspace $Y$ of $X$ such that $d(Y,\ell_2^n) \leq 1+\epsilon$.
\end{thm}

A proof of this (and more detailed discussion of the statement) can be found in, say, \cite{albiac}.

This shows that any infinite-dimensional space that satisfies the conclusion of Question 
\ref{quest:u} satisfies the conclusion of Question \ref{quest:v}. Indeed, suppose we have a finite subset $Y$ of $\ell_2$. Then $Y$ lies in some $\ell_2^n$ for $n \leq |Y|$. So, if we choose some subspace $Z$ of $X$ for which $d(Z,\ell_2^n) \leq 1+\epsilon$, we obtain a $1+\epsilon$ distortion embedding of $Y$ into $Z$.

We will also use certain convexity properties of Banach spaces. Suppose that $X$ is a Banach space. We say that $X$ is \emph{strictly convex} if, for any $x \neq y \in S_X$ we have that $\|x+y\| < 2$. A quantized version exists: we say that $X$ is \emph{uniformly convex} if for all $\epsilon > 0$ there is some $\delta>0$ such that whenever $\|x-y\| = \epsilon$ we have that $\|x+y\| \leq 2 - 2 \delta$.

The most useful interpretation of this will be geometric for us; the idea of strict convexity is simply the idea that the unit sphere contains no straight line segments. We will use the following consequence of this property: a space that is strictly convex has unique midpoints, ie, for all $x,y,z \in X$ if $\|x-z\| = \|y-z\| = \frac{1}{2} \|x-y\|$, then $z = \frac{x+y}{2}$.

In the sequel we will also be interested in defining 2-dimensional Banach spaces. We will use the term \emph{convex curve} to mean a simple closed curve in $\mathbb{R}^2$ whose interior is a convex set. We will also use the idea of a \emph{symmetric curve}: a curve $\Gamma$ is symmetric if $x \in \Gamma \iff -x \in \Gamma$. The main idea is that we take a convex symmetric curve whose interior contains the origin. This defines a Banach space through its Minkowski functional as follows:
\begin{thm}
Suppose that $\Gamma$ is a convex symmetric curve in $\mathbb{R}^2$ and let $\overline{\Gamma}$ be the curve together with its interior. Then the Minkowski functional associated to $\Gamma$, $p_\Gamma(x) = \inf\{ \lambda \in \mathbb{R}^+ : x \in \lambda \overline{\Gamma}\}$ defines a norm on $\mathbb{R}^2$ whose unit sphere is $\Gamma$.
\end{thm}
We will implicitly use this later, where we will define a norm simply by specifying such a $\Gamma$.
\subsection{Cayley Menger Determinants}\label{sec:cmdet} The \emph{Cayley-Menger determinant} can be used to classify which finite metric spaces do or do not embed into Euclidean space. The interested reader should see \cite{swane} for a readable account of the theory. We will write CMDet for the Cayley-Menger determinant, and it is defined as follows: let $y_1, y_2,\dots$ be distinct points in a space $Y$ and write $\rho_{ij} = (d_Y(y_i,y_j))^2$. Then we write $$ \text{CMDet}(y_1,\dots,y_n) = \det \left( \begin{array}{cc} P & 1 \\ 1^T & 0 \end{array}\right) $$where $P$ is the $n\times n$ matrix $P_{ij} = \rho_{ij}$ and 1 is the column vector of length $n$ consisting entirely of 1's. The following theorem is the aforementioned sharp criterion:
\begin{thm}[\cite{swane}, Theorem 5]
A metric space consisting of $n+1$ distinct points $y_0,\dots,y_n$ can be embedded isometrically into $\ell_2^n$ as an affinely independent set if and only if the sign of CMDet$(y_0,\dots,y_k)$ is $(-1)^{k+1}$ for each $k = 1,2,\dots,n$.
\end{thm}

\section{Affinely Independent Subsets of \texorpdfstring{$\ell_2$}{l2}}\label{sec:brou}
The main purpose of this section is to prove Theorem 1.3, ie, that affinely independent subsets of $\ell_2$ embed isometrically into every infinite-dimensional Banach space. In light of Dvoretzky's theorem, if we show that every affinely independent subset of $\ell_2$ embeds isometrically into small perturbations of $\ell_2^n$ for suitable $n$ then we will be done. More precisely, we prove the following:

\begin{thm}\label{thm:fixed}
Suppose $Y$ is a finite, affinely independent subset of $\ell_2$. Then there is some $\delta_Y > 0$ such that if $X$ is a Banach space with $d(X,\ell_2^{|Y|-1}) < 1 + \delta_Y$ then $Y$ embeds isometrically into $X$.
\end{thm}

This proof follows an idea of Swanepoel \cite{swane}, where he used a similar proof to show that all infinite-dimensional Banach spaces contain arbitrarily large finite equilateral sets. Theorem 3.1 was proven by Shkarin in \cite{shkarin}, however the proof method used here is different, and we include it as it may be of independent interest.

\begin{proof}
Let $Y$ be an $n+1$ point affinely independent subset of $\ell_2$. We may assume, without loss of generality, that $Y$ is an $n+1$ point subset of $\ell_2^n$, with the first point of $Y$ being zero. Let $d_Y$ be the Euclidean metric restricted to $Y$.

Label the points of $Y$ as $p_0,p_1,\dots,p_n$. Applying the result quoted in Section 2.3, we get that, for $k = 1,\dots,n$, CMDet$(p_0,\dots,p_k)$ has sign $(-1)^{k+1}$. Since $Y$ is affinely independent, no three of its points are collinear. Hence, by the strict convexity of $\ell_2$ we get that the metric space $(Y,d_Y)$ is concave in the sense of Section 2.1. 

Let $Q$ be a set of $n+1$ distinct points $q_0,\dots,q_n$ and consider a symmetric function $d:Q\times Q \rightarrow \mathbb{R}$, such that $d(q_i,q_i) = 0$ for each $i$. Then, since $(Y,d_Y)$ is concave, there is an $\epsilon > 0$ such that if $|d(q_i,q_j) - d_Y(p_i,p_j)| < \epsilon$ for all $i,j$, then $d$ is a metric on the set $Q$. Moreover, by the continuity of the determinant operation, if $\epsilon$ is sufficiently small this metric also satisfies that the sign of CMDet$(q_0,\dots,q_k)$ is $(-1)^{k+1}$ for $k=1,\dots,n$.

Suppose $Z$ is a Banach space on $\mathbb{R}^n$ with norm $\|.\|_Z$ that satisfies $\|x\|_Z \leq \|x\|_2 \leq (1+\delta_Y) \|x\|_Z$, where we choose $\delta_Y$ later. Let $\{e_1,\dots,e_n\}$ be the standard basis for $\mathbb{R}^n$. We will construct an isometry $f:Y\rightarrow Z$ by inductively defining the sequence $f(p_i)$. Begin by setting $f(p_0) = 0$ and $f(p_1) = \frac{d_Y(p_0,p_1)}{\|e_1\|_Z} e_1$, and suppose that for some $1 \leq k < n$ we have defined $f(p_0),\dots,f(p_k)$, and that $f(p_i) \in $ span$\{e_1,\dots,e_{i}\}$.

Fix $\epsilon_0,\dots,\epsilon_k \in [0,\epsilon]$, and set $\boldsymbol{\epsilon} = (\epsilon_0,\dots,\epsilon_k)$. We now define a metric space $Q(\boldsymbol{\epsilon})$ consisting of $k+2$ distinct points $q_0,\dots,q_{k+1}$ with metric $d = d_{\boldsymbol{\epsilon}}$ defined as follows:
\begin{itemize}
	\item $d(q_i,q_j) = \|f(p_i) - f(p_j)\|_2$ for each $i,j \in \{0,\dots, k\}$
	\item $d(q_i,q_{k+1}) = \|p_i - p_{k+1}\|_2 + \epsilon_i$ for each $i \in \{0,\dots,k\}$
\end{itemize}
Provided that $\delta_Y$ is sufficiently small, we have that $|d(q_i,q_j) - d_Y(p_i,p_j)| < \epsilon$ for all pairs $i,j$. Hence, by choice of $\epsilon$, $d$ is indeed a metric on $Q(\boldsymbol{\epsilon})$ and by Theorem 2.3, there is an isometric embedding $G: Q(\boldsymbol{\epsilon}) \rightarrow \ell_2^{k+1}$, onto an affinely independent set. After a series of suitable reflections we may assume that $G(q_i) = f(p_i)$ for $0 \leq i \leq k$. We then set $g(\boldsymbol{\epsilon}) = G(q_{k+1})$. Note that there are two choices for $g(\boldsymbol{\epsilon})$, depending on the sign of the $e_{k+1}$ co-ordinate. By fixing a choice, we obtain a continuous function $g:[0,\epsilon]^{k+1} \rightarrow \ell_2^{k+1}$.

Now define $\varphi:[0,\epsilon]^{k+1} \rightarrow [0,\epsilon]^{k+1}$ by $\varphi(\epsilon_0,\dots,\epsilon_k) = (y_0,\dots,y_k)$ by $$y_i = \|p_{k+1} - p_i\|_2 + \epsilon_i - \|g(\boldsymbol{\epsilon}) - f(p_i)\|_Z \, . $$

We claim that, for $\delta_Y$ sufficiently small, $\varphi$ is well defined. To see that $y_i \geq 0$, we note that $y_i \geq \|p_{k+1} - p_i\|_2 + \epsilon_i - \|g(\boldsymbol{\epsilon}) - f(p_i)\|_2 = \|p_{k+1} - p_i\|_2 + \epsilon_i - \|G(q_{k+1}) - G(q_i)\|_2 = 0$, where we have used that $\|x\|_2 \leq \|x\|_Z$ for all $x \in \mathbb{R}^n$.

On the other hand, we have that $y_i \leq \|p_{k+1} - p_i\|_2 + \epsilon_i - \frac{1}{1+\delta_Y} \|g(\boldsymbol{\epsilon}) - f(p_i)\|_2 = d(q_{k+1},q_i) - \frac{1}{1+\delta_Y} \|G(q_{k+1}) - G(q_i)\|_2 = \frac{\delta_Y}{\delta_Y + 1} d(q_{k+1},q_i) < \epsilon$, provided $\delta_Y$ is sufficiently small, where we have used that $\|x\|_Z \leq \frac{1}{1+\delta_Y} \|x\|_2$ for all $x \in \mathbb{R}^n$.

So, $\varphi$ is a continuous map from a compact subset of $\mathbb{R}^n$ to itself, and by the Brouwer fixed point theorem, it has a fixed point. Setting $f(p_{k+1})$ to be $g(\alpha)$, where $\alpha$ is the fixed point, gives the necessary extension. This completes the inductive step and hence the proof.
\end{proof}

\subsection{Some Examples} As promised in the introduction we will look at some examples to show that there is no possible way that the localised technique used in the proof of Theorem 3.1 can work for affinely dependent sets. These examples are essentially based on a modification of an idea stated in Section \ref{prereq}, where we discussed the notion of metric midpoints. It is evident that if we have a metric midpoint $z$ of two points $x,y$ and $f$ is an isometry, then $f(z)$ is a metric midpoint of $f(x),f(y)$. Moreover, if we have an embedding of two points $x,y \in X$ into $Y$, a space where metric midpoints are unique, then if we wish to extend the embedding to a metric midpoint of $x,y$ there is only one possible place that the metric midpoint can embed to.

So, in the Banach space setting, suppose we have strictly convex spaces $X,Y$ and an isometry $f:A \rightarrow Y$ where $A \subset X$ is a finite subset. Then given $a,b \in A$, and we wish to extend the isometry to the set $A^\prime = A \cup \{\frac{a+b}{2}\}$, this extension \emph{must} map the point $\frac{a+b}{2}$ to the point $\frac{f(a)+f(b)}{2}$.

This generalises to the case of general affine combinations of two points, namely $\alpha a + (1-\alpha)b$ has to map to $\alpha f(a) + (1-\alpha) f(b)$, for any $\alpha \in \mathbb{R}$.

Using this idea, we can show various obstructions to possible extensions of Theorem \ref{thm:fixed}. In the following we will say that $\tilde{f}$ is \emph{an extension} of $f$ if $\tilde{f}$ is defined on a larger set than $f$, and $f = \tilde{f}|_{\text{dom} f}$.
\begin{thm}\label{thm:counter}
\begin{enumerate}
	\item There is a 3 points subset $Y$ of $\ell_2^2$ and a point $y \in \ell_2^2$, such that for all $\epsilon > 0$ there is a Banach space $X$ with $d(X,\ell_2^2) \leq 1+\epsilon$, and an embedding $f:Y \rightarrow X$ which is isometric, but that there is no extension $\tilde{f}: Y \cup \{y\} \rightarrow X$ which remains isometric.
	\item There is a three-point subset $Y$ of $\ell_2^2$ such that for any strictly convex Banach space $X$ that contains no subspace isometric to $\ell_2^2$, and for every isometric embedding $f:Y \rightarrow X$ there is a point $y \in \ell_2^2$ such that no extension $\tilde{f}:Y \cup \{y\} \rightarrow X$ remains isometric.
	\item Suppose that $X$ is a finite-dimensional, strictly convex Banach space that contains no 2-dimensional subspace isometric to $\ell_2^2$. Then there is a finite subset $Y$ of $\ell_2^2$ that does not embed into $X$.
\end{enumerate}
\end{thm}

The first part of Theorem 3.2 is saying that the proof of Theorem 3.1 does not admit an obvious extension to the case of affine dependence.

The second part is saying that the obstruction in the first part of the theorem is, in some sense, typical. A modification of the first example allows us to block extensions of the type used in the proof of Theorem 3.1 to any strictly convex space with no subspaces isometric to $\ell_2^2$.

The third part has two interpretations. The first one is saying that if $X$ is a strictly convex finite-dimensional Banach space with no 2-dimensional subspace isometric to $\ell_2^2$, then there is a least $M$ such that there is an $M$-point subset of $\ell_2^2$ that does not embed isometrically into $X$. This is an isometric invariant of the space. For the second interpretation, note that by Theorem 3.1, for every $n \in \N$ there is an $N \in \N$ such that every affinely independent subset of $\ell_2^n$ embeds into every $N-$dimensional Banach space sufficiently close to $\ell_2^N$ (indeed, $N=n+1$ in this case.) Part (3) shows that this fails without the condition of affine independence, even in the case $n=2$.

\begin{proof}[Proof of \ref{thm:counter} part (1)]
Let $Y$ be the subset $\{(0,0),(1,1),(1,0)\}$ of $\ell_2^2$ and let $y = (0.5,0)$.
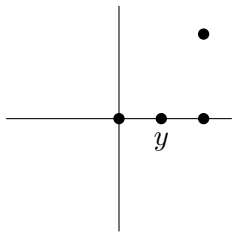
\begin{figure}[ht]

\begin{tikzpicture}[scale=1.5]
	\draw (1,0) -- (-1,0);
	\draw (0,-1) -- (0,1);
	
	\fill[black] (0,0) circle (0.05cm);
	\fill[black] (0.75,0.75) circle (0.05cm);
	\fill[black] (0.75,0) circle (0.05cm);
	
	\fill[black] (0.375,0) circle (0.05cm);
	
	\draw (0.375,-0.2) node {$y$};

\end{tikzpicture}
\caption{The subset $Y$ and the point $\{y\}$}
\end{figure}

We now construct a Banach space $X$ for which $d(X,\ell_2^2)$ is small. Fix some $\epsilon > 0$. We construct a symmetric convex curve $\Gamma$ as follows: Let $x = (x_1,x_2) \in \R^2$.
\begin{itemize}
	\item If the argument of $x$ is not in $[\pi/4 - \epsilon, \pi/4 + \epsilon] \cup [5 \pi/4 - \epsilon, 5\pi/4 + \epsilon]$, and $x \in S_{\ell_2^2}$, then $x \in \Gamma$.
	\item If the argument of $x$ is in $[\pi/4 - \epsilon, \pi/4+\epsilon]$ we set $x \in \Gamma$ if $C(\epsilon) (x_1^{3/2} + x_2^{3/2})^{2/3} = 1$, where $C(\epsilon)$ is chosen such that the curve $\Gamma$ is continuous. We think of this sector of the unit disc as the 'bad' sector.
	\item If the argument of $x \in [5 \pi/4 - \epsilon, 5 \pi/4 + \epsilon]$ we say $x \in \Gamma$ just when $-x \in \Gamma$.
\end{itemize}
The norm that this Banach space is equipped with is evidently strictly convex. Moreover, $d(X,\ell_2^2) \rightarrow 0$ as $\epsilon \rightarrow 0$.

We now embed the subset $Y$ of $\ell_2^2$ into the Banach space $X$ in the following way: send $(0,0)$ to $(0,0)$ and $(1,0)$ to $(1,0)$. Then, if $\epsilon$ is sufficiently small, we can embed $(1,1)$ to some point $z$ whose argument is in $[\pi/4 - \epsilon, \pi/4+\epsilon]$ and which lies on $\sqrt{2} S_X$. Thus $z$ lies in the 'bad' sector, ie, the sector where the norm is not equal to the $\ell_2^2$ norm.

By the strict convexity of $X$, and the remarks at the beginning of the section concerning midpoints, the point $y$ must map to the point $(0.5,0)$ in $X$.

We now just need to observe that provided $\epsilon$ is small enough the vector in $X$ joining $(0.5,0)$ to $z$ is in the 'good' sector, ie, will have the $\ell_2$ norm. However, the $\ell_2$ norm of the vector joining $(0.5,0)$ to $z$ is not the $\ell_2$ norm of the vector joining $(0.5,0)$ to $(1,1)$ (this can be seen geometrically: the argument of the vector joining $(1,0)$ to the point $z$ is not $\pi/2$.)
\end{proof}
\begin{proof}[Proof of \ref{thm:counter} part (2)]
Let $Y$ be the three-point subset of $\ell_2^2$ given by $\{(0,0), (1,0), (0,1)\}$ and let $X$ be a strictly convex Banach space satisfying the assumptions in the theorem. Let $f$ be an isometric embedding of $Y$ into $X$. Note that, without loss of generality, we may assume $f(0,0) = (0,0)$.

Now, suppose that for each $y \notin Y$ we can extend the isometric embedding $f:Y \rightarrow X$ to an embedding $\tilde{f}:Y \cup \{y\} \rightarrow X$. In particular, we can take $y = \alpha(1,0) + (1-\alpha) (0,1)$. By the remarks preceding this theorem, the extension $\tilde{f}$ must map $y$ to $\alpha f(1,0) + (1-\alpha) f(0,1)$ for any $\alpha \in \mathbb{R}$. Thus we have that $\|\alpha f(1,0) + (1-\alpha) f(0,1)\| = \sqrt{\alpha^2 + (1-\alpha)^2}$. By homogeneity it follows that $f(1,0)$ and $f(0,1)$ span a subspace isometric to $\ell_2^2$, contradicting the assumptions in the theorem.
\end{proof}
\begin{proof}[Proof of \ref{thm:counter} part (3)]
Suppose $X$ is a finite-dimensional, strictly convex Banach space with no subspace isometric to $\ell_2^2$, and suppose that every finite subset of $\ell_2^2$ embeds into $X$.

Let $s_1,s_2,\dots$ be a dense subset of $\ell_2^2$ with $s_1 = 0$. For each $n$ let $f_n:\{s_1,\dots,s_n\} \rightarrow X$ be an isometric embedding, such that (without loss of generality) $f_n(s_1) = s_1$.

By a diagonal argument, we may assume that $\lim_n f_n(s_j)$ converges for each $j$, and call this limit $f(s_j)$.By the continuity of the norm, $f$ extends to an isometric embedding 	$\hat{f} : \ell_2^2 \rightarrow X$, which is linear by the remarks preceding the theorem. This contradicts the assumptions in the theorem.
\end{proof}

\section{Some Partial Results for Question 1.1}
\subsection{Some Positive Results}
We start this section by presenting some positive results for Question 1.1, which will lead to further partial answers to Question 1.2.

It is trivial that finite-dimensional spaces satisfy the conclusion of Question 1.1 (via a compactness argument). It is also trivial that spaces that are isometrically universal for all finite metric spaces satisfy the conclusion of Question 1.1 (eg, $\ell_\infty$, $C[0,1]$.)

A key result for less obvious examples is a result due to Ball in \cite{ball}:
\begin{thm}
If $1 \leq p \leq \infty$, and $Y$ is an $n$ point subset of $L_p$, then $Y$ isometrically embeds into $\ell_p^{\binom{n}{2}}$.
\end{thm}
This gives that $L_p$ spaces (and hence $\ell_p$) satisfy the conclusions of Question \ref{quest:u}.

Summing up, we get that: 
\begin{thm}
The following Banach spaces satisfy the conclusions of Question \ref{quest:u}:
\begin{itemize}
	\item All finite-dimensional Banach spaces.
	\item Any Banach space that is isometrically universal for all finite metric spaces.
	\item $L_p$ spaces.
\end{itemize}
\end{thm}

\begin{proof}
Only the third bullet point needs justification. Suppose $X$ is an $L_p$ space. By the first bullet point we may assume that $X$ is infinite-dimensional, and hence contains $\ell_p^N$ isometrically for each $N \in \mathbb{N}$. Fix an $n$-point metric space $Y$, and suppose that for each $\epsilon > 0$ there is an embedding of $Y$ into $X$ with distortion $1+\epsilon$, which we denote $Y_\epsilon$. By Theorem 4.1, we have that $Y_\epsilon$ is isometric to a subset of $\ell_p^{\binom{n}{2}}$, ie, there is an embedding with distortion $1+\epsilon$ of $Y$ into $\ell_p^{\binom{n}{2}}$. By compactness, $\ell_p^{\binom{n}{2}}$ contains an isometric copy of $Y$, and since $\ell_p^{\binom{n}{2}}$ embeds isometrically into $X$ we are done.
\end{proof}

An old result of Fr\'{e}chet states that every $n$-point metric space embeds into $\ell_\infty^n$. Thus the class of Banach spaces in the second bullet point of Theorem 4.2 includes spaces that contain $\ell_\infty^n$ isometrically for all $n$. It is thus natural to consider the class of Banach spaces that for each $\epsilon > 0, n \in \N$ contain a subspace $Z$ such that $d(Z,\ell_\infty^n) < 1 + \epsilon$. Then Question 1.1 asks whether such a space is isometrically universal for finite metric spaces. The answer to this question is negative in general - see Example 4.10. However, the following result shows that if a Banach space contains almost isometric copies of $\ell_\infty^n$ for every $n$ then it \emph{does} contain isometric copies of every finite concave metric space. The method resembles the proof of Theorem 3.1, where we construct a family of embeddings depending on some parameter and then use the Brouwer fixed point theorem to show that one of the embeddings is isometric.
\begin{thm}\label{thm:cotype}
Suppose $X$ is a Banach space which contains, for each $n \in \N, \epsilon > 0$ a subspace $Z$ such that $d(Z,\ell_\infty^n) \leq 1+\epsilon$. Suppose that $Y$ is a finite metric space which is concave. Then $X$ contains an isometric copy of $Y$.  
\end{thm}

The proof of this result is similar to a proof of a result due to Swanepoel.

\begin{proof}
As in the proof of Theorem 3.1, we show that given any finite metric space $Y$ that is concave there is a parameter $\epsilon$ such that if $|Y| = n$ and $d(\ell_\infty^n,Z) \leq 1+\epsilon$ then $Y$ embeds isometrically into $Z$.

Since $Y$ is concave, we may fix $\eta > 0$ such that $d(x,y) + d(y,z) - d(x,z) > 2 \eta$ for each triple $x,y,z$ of pairwise distinct points in $Y$. Suppose that $d(Z,\ell_\infty^n) \leq 1 + \delta$, and assume $Z$ is $\mathbb{R}^n$ with norm $\|.\|_Z$ satisfying $\|x\|_Z \leq \|x\|_\infty \leq (1+\delta) \|x\|_Z$, where $\delta > 0$ is to be determined.

For $1\leq j < i \leq n$ consider $0 \leq \epsilon_{ij} \leq \eta$ and set $\epsilon$ to be the vector of all the $\epsilon_{ij}$'s. Let $x_1,\dots,x_n$ be the points of $Y$ and consider the following points of $\mathbb{R}^n$: $$p_1(\epsilon) = (d(x_1,x_1), d(x_1,x_2), d(x_1,x_3), \dots , d(x_1,x_n)) $$ $$ p_2(\epsilon) = (d(x_2,x_1) + \epsilon_{2,1}, d(x_2,x_2), d(x_2,x_3), \dots , d(x_2,x_n) )$$ $$ p_3(\epsilon) = (d(x_3,x_1) + \epsilon_{3,1}, d(x_3,x_2) + \epsilon_{3,2} , d(x_3,x_3) \dots , d(x_3,x_n))$$ $$\vdots$$ $$p_n(\epsilon) = (d(x_n,x_1)+ \epsilon_{n,1}, \dots, d(x_n,x_{n-1}) + \epsilon_{n,n-1}, d(x_n,x_n))$$

(Note that $x_i \mapsto p_i(0)$ is the Fr\'{e}chet embedding of $Y$ into $\ell_\infty^n$.)

We have, for $i \neq j$, that $$\|p_i(\epsilon) - p_j(\epsilon)\|_\infty = \sup_k |d(x_i,x_k) +\epsilon_{i,k} - (d(x_j,x_k) + \epsilon_{j,k})|$$where we set $\epsilon_{kl} = 0$ for $l \geq k$. This is equal to $d(x_i,x_j) + \epsilon_{ij}$, as for $k \neq i,j$ we have $$d(x_i,x_k) - d(x_j,x_k) + \epsilon_{ik} - \epsilon_{kj} \leq d(x_i,x_j) - 2 \eta + \epsilon_{ik} - \epsilon_{jk} \leq d(x_i,x_j) \, .$$

Set $\varphi_{ij}(\epsilon) = d(x_i,x_j) + \epsilon_{i,j} - \|p_i(\epsilon)- p_j(\epsilon)\|_Z$, and $\varphi(\epsilon) = (\varphi_{ij}(\epsilon))_{1 \leq j < i \leq n}$. We have that $\varphi_{ij}(\epsilon) \geq 0$ by the computation of $\|p_i(\epsilon) - p_j(\epsilon)\|_\infty$, and the fact that the $\ell_\infty$-norm majorizes the $Z$-norm. Now, $\varphi_{i,j}(\epsilon) \leq d(x_i,x_j) + \epsilon_{ij} - \frac{1}{1+\delta} (d(x_i,x_j) + \epsilon_{ij}) = (d(x_i,x_j) + \epsilon_{ij}) \frac{\delta}{\delta+1}$. So if $\frac{\delta}{\delta+1}$ is small enough that this is less than $\eta$ then $\varphi$ sends $[0,\eta]^N \rightarrow [0,\eta]^N$ (where $N = \binom{n}{2}$).

As $\varphi$ is continuous, by the Brouwer fixed point theorem we have some $\epsilon$ such that $\varphi(\epsilon) = \epsilon$, which satisfies $\|p_i(\epsilon) - p_j(\epsilon)\|_Z = d(x_i,x_j)$ for all $i,j$.
\end{proof}

\begin{rem}
Spaces that contain $\ell_\infty^n$ almost isometrically are those with no non-trivial cotype. Examples include Tsirelson's original space and $(\oplus \ell_{p_n}^{q_n})_2$ where $p_n,q_n\rightarrow \infty$.
\end{rem}

Recall, by Dvoretzky's theorem any infinite-dimensional space that satisfies the conclusion of Question 1.1 satisfies the conclusion of Question 1.2. Theorem 4.2 thus furnishes a list of infinite-dimensional spaces that satisfy the conclusion of Question 1.2.

Moreover, combining Dvoretzky's theorem, and Theorem 4.1 (due to Ball) we can arrive at the following result.
\begin{thm}
Suppose that $X$ is an infinite-dimensional Banach space such that for some $p \in [1,\infty]$, $\ell_p^m$ isometrically embeds into $X$ for each $m$. Then $X$ satisfies the conclusion of Question 1.2.
\end{thm}
\begin{proof}
If $p = \infty$ in the statement of the theorem, then $X$ is universal for all finite metric spaces, and thus we are done. The proof of the case $p \in [1,\infty)$ is very similar to the proof of the third bullet point of Theorem 4.2, however we spell it out for clarity. Let $Y$ be an $n$-point subset of $\ell_2$, which without loss of generality we may assume is an $n$-point subset of $\ell_2^n$, and let $\epsilon > 0$. By Dvoretzky's theorem there is some $N$ such that $\ell_2^n$ embeds into $\ell_p^N$ with distortion $1+\epsilon$. So $Y$ embeds into $\ell_p^N$ with distortion $1+\epsilon$ and by Ball's result it embeds into $\ell_p^{\binom{n}{2}}$ with distortion $1+\epsilon$. By compactness there is an isometric embedding of $Y$ into $\ell_p^{\binom{n}{2}}$, and hence into $X$.
\end{proof}

\begin{rem}
We note that spaces of the type in the previous theorem include infinite-dimensional $L_p$ spaces and any $\ell_p$ direct sum of a sequence of Banach spaces. The latter is of note as many examples in the next section are of this type.
\end{rem}

\subsection{Some Negative Results:} In this section we construct some spaces that do not satisfy the conclusion of Question 1.1. The significance of this statement is that any attempt to prove that "every infinite-dimensional Banach space satisfies the conclusion of Question 1.2" can not prove the stronger statement "every infinite-dimensional Banach space satisfies the conclusion of Question 1.1".

\begin{example}
Let $X$ be the space $X = (\oplus_{n=1}^\infty \ell_{1+1/n})_2$. $X$ fails to satisfy the conclusions of Question 1.1. The space $X$ is evidently strictly convex, and thus has unique metric midpoints. Now consider $S = \{(0,0), (1,0), (0,1) , (1,1)\} \subset \ell_1^2$. This is a subset such that both $(1,0)$ and $(0,1)$ are metric midpoints of $(0,0)$ and $(1,1)$. Hence  $S$ does not isometrically embed into $X$. However, the formal identity mapping from $\ell_{1+1/n}^2$ to $\ell_1$ has norm $2^{1 - 1/(1-1/n))} \approx 2^{1/n}$, which converges to 1 as $n \rightarrow \infty$, so there exist $1+\epsilon$ distortion embeddings of $S$ into $X$ for each $\epsilon > 0$.
\end{example}

Note that the space $X$ evidently satisfies the conclusion of Question 1.2 as it contains a copy of $\ell_2$ isometrically. (If we switched 2 for $p$ in the definition of $X$ we would still have that $X$ satisfied the conclusion of Question 1.2 by Remark 4.6.)

We can extend the construction of the space $X$ to give a class of spaces that fail the conclusion of Question 1.1 as follows:
\begin{prop}\label{prop:struni}
Suppose that $X$ is a strictly convex Banach space that is not uniformly convex. Then it does not satisfy the conclusion of Question \ref{quest:u}.
\end{prop}
\begin{proof}
Spelling out the property of $X$ not being uniformly convex: there is an $\epsilon \in (0,2]$ such that for each $n$ there are points $x_n,y_n$ with $\|x_n - y_n\| = \epsilon$ and $\|\frac{x_n+y_n}{2}\| \geq 1 - \frac{1}{n}$. Consider the following metric on 4 points $\{a,b,c,d\}$ where $d(a,b) = d(a,c) = d(a,d) = 1$ and $d(b,c) = d(c,d) = \epsilon/2$ and $d(b,d) = \epsilon$. There exist almost isometric embeddings of this metric space into $X$: simply send $a$ to zero, $b,d$ to $x_n,y_n$ respectively and $c$ to $\frac{x_n+y_n}{2}$.

However, there is no isometric embedding. Indeed, if there was one, then without loss of generality $a$ would map to zero, and $b,d$ would map to $S_X$ with $\|b-d\| = \epsilon$. Since isometries into strictly convex metric spaces preserve metric midpoints $c$ would have to map to $(b+d)/2$, but this would be of distance $< 1$ from zero, by strict convexity.
\end{proof}

The formulation in Proposition 4.8 allows us to give an example of a space that is important in PDE theory that fails the conclusion of Question 1.1. To state this we will need to go into properties of Orlicz spaces taken from \cite{orlicz}. We will take the notation from that paper, namely, the function $\varphi$ is the function generating the Orlicz space, $\Phi(t) = \int_0^t \varphi(t^\prime) dt^\prime$, $\psi(t) = \varphi^{-1}(t)$ and $\Psi(t) = \int_0^t \psi(t^\prime) \, d(t^\prime)$. We will only deal with continuous Orlicz functions where these things are well defined, and all continuous.\footnote{This simplifies what is about to be stated, but not by much.}

The results of the paper show the following: Let $V$ denote the supremum attained by the function $\varphi$. Then the Orlicz space generated by the function $\varphi$ is strictly convex if and only if $\psi,\Psi$ both tend to infinity as $v \rightarrow V$. Moreover, the paper also shows that the Orlicz space generated by the function $\varphi$ is uniformly convex just when $\frac{\Phi(2u)}{\Phi(u)}$ is bounded above, and for all $\epsilon \in (0,0.25)$, $\frac{\Phi(u)}{\Phi((1-\epsilon)u)}$ is bounded below.

In the case of $\varphi(t) = e^t - 1$, the Orlicz space corresponding to the function $\varphi$ is called the exponential Orlicz space. In this case, $\Phi(t) = e^t - t - 1$, and $\frac{\Phi(2t)}{\Phi(t)} = \frac{e^{2t} - 2t - 1}{e^t - t - 1} \approx e^t$, which is definitely not bounded above. Thus, the space generated is not uniformly convex. However, in this case $\psi(t) = \log (t+1)$ and $\Psi(t) = (t+1)\log(t+1) - t$, which are both functions that tend to $\infty$ as $t \rightarrow \infty$. We thus have the following.

\begin{example}
The exponential Orlicz space fails the conclusions of Question \ref{quest:u}.
\end{example}

We end this section with a further example whose purpose is twofold. On the one hand it shows that Theorem 4.3 does not generalize. It also sheds further light on the permanence properties of Banach spaces that satisfy the conclusion of Question 1.1.

\begin{example}
Let $p_n \rightarrow \infty$ and let $Y$ be the Banach space $(\oplus_{n=1}^\infty \ell_{p_n}^n)_2$. The space $Y$ is strictly convex, but not uniformly convex, so by Proposition 4.8 it does not satisfy the conclusion of Question 1.1. However, $Y$ contains $\ell_\infty^n$ almost isometrically for each $n$ - this is the space whose existence we promised in the remarks before Theorem 4.3.
\end{example}

This space also allows us to show that several possible permanence properties of spaces that satisfy the conclusion of Question 1.1 fail. Let $X = (\oplus \ell_\infty^n)_2$. Then $d(X,Y) \leq \prod n^{1/p_n}$, which we can make as small as we please. This show that if $X$ satisfies the conclusion Question 1.1, there is not necessarily some $\epsilon > 0$ such that $d(X,Y) \leq 1+\epsilon$ implies that $Y$ satisfies the conclusion of Question 1.1.

\begin{rem}
Indeed, we can construct two spaces $X,Y$ such that $d(X,Y) = 1$ such that $X$ satisfies the conclusion of Question 1.1 and $Y$ does not. Let $d_n$ be a sequence that takes every $m \in \N$ infinitely often, and $p_n \rightarrow \infty$. Then if $Y = (\bigoplus_{m,n} \ell_{p_n}^{d_m})_2$ and $X = (\bigoplus \ell_\infty^n)_2 \bigoplus_2 Y$, then we can prove that $d(X,Y) = 1$. It is evident that $X$ satisfies the conclusion of Question 1.1, and Proposition 4.8 allows us to see that $Y$ does not.
\end{rem}

\section{An Application and Some Open Questions} We give an application of Theorem 3.1 that may be of independent interest. Recall that the $\alpha$-snowflake of a finite metric space $(X,d)$ is the metric space $(X,d^\alpha)$ with $\alpha \in (0,1)$.

\begin{thm}
Suppose that $(X,d)$ is a finite metric space, and $Y$ is an infinite-dimensional Banach space. Then there exists $\alpha^\prime \in (0,1)$ such that whenever $0 < \alpha < \alpha^\prime$, the space $(X,d^\alpha)$ isometrically embeds into $Y$.
\end{thm}

The proof of this consists of combining two ingredients, the first being Theorem 3.1 of this paper. The second ingredient is a result of Le Donne, Rajala and Walsberg in \cite{snow}. Proposition 2.2 in this paper shows that for $\alpha$ sufficiently close to zero, the space $(X,d^\alpha)$ embeds into $\ell_2^{|X|}$ isometrically.

\begin{proof}
Let $x_1,\dots,x_n$ be the points of $X$. The proof of Proposition 2.2 in \cite{snow} shows that, for $\alpha$ sufficiently small, there is an embedding that takes the points of $X$ to points close to $\frac{1}{\sqrt{2}} e_1, \frac{1}{\sqrt{2}} e_2,\dots, \frac{1}{\sqrt{2}} e_n$. More precisely, if we let the isometric embedding be denoted by $f_\alpha : (X,d^\alpha) \rightarrow \ell_2^{|X|}$ we have that $f_\alpha(x_i) \rightarrow \frac{1}{\sqrt{2}} e_i$, as $\alpha \rightarrow 0$.

Let $\alpha^\prime > 0$ be such that $\forall \alpha < \alpha^\prime$ we have $\|f_\alpha(x_i) - \frac{1}{\sqrt{2}} e_i\| < \frac{1}{100n}$. Then, for each $\alpha \in (0,\alpha^\prime)$ we have a basis $\frac{1}{\sqrt{2}} e_i$, and a sequence $f_\alpha(x_i)$ such that $\sum_{i=1}^n \|f_\alpha(x_i) - \frac{1}{\sqrt{2}} e_i\| < \frac{1}{100}$.

It follows easily that the $f_\alpha(x_i)$ are linearly independent, and so we are in a position to apply Theorem 3.1.
\end{proof}

\subsection{Open Questions:} 
The main question that remains open is the following:
\begin{quest}
	Let $X$ be an infinite-dimensional Banach space, and $A$ a finite subset of $\ell_2$. Then does $A$ isometrically embed into $X$?
\end{quest}
Our paper establishes this result is true whenever $A$ is affinely independent, or when $X$ belongs to various natural classs of Banach spaces (eg, any $p$-direct sum of Banach spaces).

For affinely dependent sets, we are able to prove the following result: Suppose $X$ is a Banach space of dimension 2 or above. Let $A = \{0,x,y,\alpha x + (1-\alpha) y\}$ where $x,y \in \ell_2^2$ and $\alpha \in [0,1]$. Then $A$ isometrically embeds into $X$. This is the first non-obvious affinely dependent set. The proof consists of constructing a 1-parameter family of embeddings of $A$ that are isometric on $\{0,x,y\}$. Then an intermediate value property is used to find one that is isometric on $A$. However, our method does not seem to generalize to more complicated affinely dependent sets.

Our work in Section 4 leads naturally to the following question:
\begin{quest}
	Let $X$ be an infinite-dimensional, uniformly convex space. Does $X$ contain isometric copies of every finite metric space it contains almost isometrically?
\end{quest}

We finally mention a generalization of Questions 1.1 and 1.2 that may be of further interest.
\begin{quest}
	Let $\mathcal{Y}$ by a class of metric spaces. We say that a Banach space $X$ has property $U(\mathcal{Y})$ if every space $Y \in \mathcal{Y}$ that almost isometrically embeds into $X$ embeds isometrically into $X$. Are there characterizations of $\{X : X$ has $U(\mathcal{Y})\}$ for natural classes of spaces $\mathcal{Y}$?
\end{quest}

Consider the following classes of metric spaces:
\begin{itemize}
	\item Let $\mathcal{Y}_1$ be the class of finite equilateral metric spaces.
	\item Let $\mathcal{Y}_2$ be the class of finite affinely independent subsets of $\ell_2$.
	\item Let $\mathcal{Y}_3$ be the class of finite subsets of $\ell_2$.
	\item Let $\mathcal{Y}_4$ be the class of finite concave metric spaces
	\item Let $\mathcal{Y}_5$ be the class of finite metric spaces.
\end{itemize}

The work of Swanepoel in \cite {swane} shows that all infinite-dimensional Banach spaces have property $U(\mathcal{Y}_1)$, and our paper shows that all infinite-dimensional Banach spaces have property $U(\mathcal{Y}_2)$. Question 5.2 is asking whether every Banach space has property $U(\mathcal{Y}_3)$.

Our work also shows that there are spaces that have property $U(\mathcal{Y}_3)$ but not property $U(\mathcal{Y}_5)$ and shows that every infinite-dimensional space with no non-trivial cotype has $U(\mathcal{Y}_4)$. 

Question 5.3 is asking whether every infinite-dimensional uniformly convex Banach space has property $U(\mathcal{Y}_5)$

\section*{Acknowledgements}
I would like to thank my PhD supervisor Andras Zsak for some helpful comments and suggestions. I would also like to thank Bill Johnson and Mikhail Ostrovskii for answering the question I initially posed on mathoverflow.net that started the idea for this paper. I would like to thank Lionel Nguyen Van Th\'{e} for bringing to my attention the paper \cite{shkarin}.

\begin{thebibliography}{10}
	\bibitem{ostrov}
	 {\sc M Ostrovskii} {\em http://mathoverflow.net/questions/221181/} Retreived last at \today
	\bibitem{albiac}
	 {\sc F Albiac, N Kalton} {\em Topics in Banach space theory} Graduate Texts in Mathematics, Springer
	\bibitem{orlicz} {\sc HW Milnes} {\em Convexity of Orlicz spaces} Pacific J. Math. 7 (1957), no. 3, 1451--1483
	\bibitem{swane} {\sc K Swanepoel} {\em Equilateral sets in finite-dimensional normed spaces} arXiv:math/0406264
	\bibitem{day} {\sc M Day} {\em Uniform Convexity in Factor and Conjugate Spaces} Annals of Mathematics Second Series, Vol. 45, No. 2 (Apr. 1944), pp. 375-385
	\bibitem{ball} {\sc K Ball} {\em Isometric embedding in $\ell_p$-spaces} European Journal of Combinatorics, Volume 11, Issue 4, July 1990, Pages 305-311
	\bibitem{weaver} Suggested by {\sc Nik Weaver} on MathOverflow
	\bibitem{snow} {\sc Enrico Le Donne, Tapio Rajala, Erik Walsberg} {\em Isometric Embeddings of snowflakes into finite-dimensional Banach spaces} arXiv:1609.03377v1 [math.MG]
	\bibitem{shkarin} {\sc S.A. Shkarin} {\em Isometric embedding of finite ultrametric spaces
in Banach spaces} Topology and its Applications 142 (2004) 13–17
\end{thebibliography}
\end{document}